\documentclass[12pt,a4paper,hypertex]{amsart}
\usepackage{amsfonts,amssymb,amsmath}
\def\gg{{\mathfrak{g}}}
\def\hh{{\mathfrak{h}}}
\def\CC{{\mathbb C}}
\def\NN{{\mathbb N}}
\def\ZZ{{\mathbb Z}}
\def\lr{\:\raisebox{-0.6ex}{$\stackrel{p_{\xi^k}}{\to}$}\:}

\def\gtimes{{\gg\otimes A}}
\def\gof{{\gg_{\rm aff}\otimes A}}
\def\gf{{\gg_{\rm aff}}}
\def\oG{{\overline{\Gamma}}}

\newtheorem{dfn}{Definition}[section]
\newtheorem{definition}[dfn]{Definition}
\newtheorem{theorem}[dfn]{Theorem}
\newtheorem{corollary}[dfn]{Corollary}
\newtheorem{lemma}[dfn]{Lemma}
\newtheorem{proposition}[dfn]{Proposition}
\newtheorem{remark}{Remark}

\begin{document}

\author{S. Eswara Rao  }
\author{Vyacheslav Futorny }
\title[Loop Kac-Moody  Algebras
\  ]{Representations of Loop Kac-Moody Lie Algebras }
\address{School of Mathematics,
 Tata Institute of Fundamental Research,
 Homi Bhabha Road, Colaba, Mumbai, 400 005, India} \email{senapati@math.tifr.res.in
 }
\address{Instituto de Matem\'atica e Estat\'\i stica,
Universidade de S\~ao Paulo, Caixa Postal 66281, S\~ao Paulo, CEP
05315-970, Brasil} \email{futorny@ime.usp.br}


\begin{abstract}

We study representations of the Loop Kac-Moody Lie algebra $\gg
\otimes A$, where $\gg$ is any Kac-Moody  algebra and $A$ is a
ring of Laurent polynomials  in $n$ commuting variables. In
particular, we study representations with finite dimensional
weight spaces and their graded versions.
When we specialize $\gg$ to be a finite dimensional or affine Lie algebra
we obtain modules for toroidal Lie algebras.
\end{abstract}

\maketitle

\section*{Introduction}
Let $\gg_0$ be a simple finite dimensional Lie algebra and $A$
  the ring of Laurent polynomials in $n$ commuting variables.
 The paper  grew  out of our interest to construct new modules
for toroidal
 Lie algebras, which are the universal central extensions of
$\gg_0 \otimes A$.  In particular, $\gg_0\otimes A$  is a quotient
of the corresponding toroidal Lie algebra and, hence, any module
of $\gg_0\otimes A$ lifts to a module of this toroidal Lie
algebra. Let $\gf$ be the the affinization of $\gg_0$. Then $\gof$
is a quotient of a toroidal Lie algebra with $n+1$ variables (see
\cite{E2}, Example 4.2). Thus, by constructing modules for
$\gg_0\otimes A$ and $\gof$, we obtain modules for corresponding
toroidal Lie algebras. In fact, it is shown in \cite{E3} that any
irreducible module over a toroidal Lie algebra with finite
dimensional weight spaces comes from some module over
$\gg_0\otimes A$ or $\gof$.

Let $\gg$ be any Kac-Moody Lie algebra, $\hh\subseteq \gg$ a Cartan
subalgebra. In this paper we study highest weight modules for
$\gtimes$ for any Kac- Moody Lie algebra $\gg$ as our methods work
in such generality. The main idea comes from the paper of Wilson
\cite{W}, where the author studies  the case of $\gg = s\ell_2$
and $A=\mathbb C[t,t^{-1}]$. We generalize most of the results of
\cite{W} for any Kac-Moody Lie algebra and for any number of
variables.

  Set $\gg' = [\gg, \gg]$ and $\hh'=\hh\cap \gg'$. Let $\gg' = N^+ \oplus
\hh'\oplus N^{-}$  be the standard decomposition.  Then
$\gg'\otimes A = (N^+ \otimes A) \oplus (\hh' \otimes A) \oplus
(N^- \otimes A)$ is a natural triangular decomposition. We study
irreducible modules $V$ for $\gg'\otimes A$ such  that $V$ admits
a nonzero element $v$ such that $(N^+ \otimes A) v = 0$ and $\hh'
\otimes A$ acts on $v$ via some function $\psi\in(\hh'\otimes
A)^*$.
  We will denote the corresponding module by $V(\psi)$. Note that $V(\psi)$ is a weight
  module, that is
$$V(\psi)=\oplus_{\mu\in (\hh')^{*}} V(\psi)_{\mu}.$$ Module
$V(\psi)$ will be called a \emph{highest weight} module with a
highest weight $\psi$ (though the construction is similar to the
construction of loop modules in the affine case).

   We are mainly
interested in  modules $V(\psi)$ having finite dimensional weight
spaces $V(\psi)_{\mu}$ for all $\mu\in (\hh')^{*}$. Necessary and
sufficient conditions for a module $V(\psi)$ to have all finite
dimensional weight spaces were given in \cite{E2}:
 there must exist a co-finite ideal $I$ of $A$ such that $$(\gg' \otimes
 I)V(\psi) = 0.$$ Billig and Zhao  defined exp-polynomial maps
(see Definition 3.1) and showed that $V(\psi)$ has finite
dimensional weight spaces if $\psi$ is an exp-polynomial map
\cite{BZ}. Rencai Lu and Zhao  proved that this condition is
necessary \cite{RZ}. They also gave an explicit formula for the
exp-polynomial map in terms of the co-finite ideal $I$.  This
explicit form  is very useful for us in the study of $V(\psi)$.

Note that $\gg'\otimes A$ is naturally $\ZZ^n$-graded with a
gradation coming from $A$.  Hence, it is natural and important to
study $\ZZ^n$-graded modules for $\gg' \otimes A$. Though module
$V(\psi)$ is not $\ZZ^n$-graded, we can define a natural
$\gg'\otimes A$-module structure on $V(\psi)\otimes A$, which is
already $\ZZ^n$-graded. Even though $V(\psi)$ is irreducible, the
module $V(\psi) \otimes A$ need not be irreducible.  It was shown
in \cite{E2} that $V(\psi)\otimes A$ is completely reducible with
finitely many irreducible components. The interplay between
modules $V(\psi)$ and irreducible components of $V(\psi)\otimes A$
is the main content of this paper.

Given any function $\psi \in (\hh' \otimes A)^*$ one can  define a
natural $\ZZ^n$- graded map $\overline{\psi}$ from $\hh' \otimes A$ to $A$ and  a
$\ZZ^n$-graded irreducible module $V(\overline{\psi})$. This
module is an irreducible component of $V(\psi)\otimes A$
(\cite{E2}). Moreover,  all irreducible components are graded
isomorphic up to a grade shift.

  Suppose $V(\psi)\otimes A$ decomposes into $R$ isomorphic copies. Then
the map $\psi$ has an interesting decomposition takes. $\psi$ can be decomposed into a sum of $R$
exp-polynomial maps (see Proposition 3.4). In addition,  each of
these exp-polynomial maps gives rise to an irreducible highest
weight module. Further $V(\psi)$ is isomorphic to a tensor product
of these modules (see Corollary 3.9). Moreover, all these
components of the tensor product are isomorphic as $\gg'\otimes
A$-modules up to an automorphism of $\gg'\otimes A$
(see Section 5). On the other hand, the restriction of any of
these automorphisms to $\gg'$ is identity. Hence, all components
of the decomposition of $V(\psi)$ into the tensor product are
isomorphic as $\gg'$-modules. In particular, they are isomorphic
as $\hh'$-modules. So the character of $V(\psi)$ is the
product of characters of one of these components (see
Proposition 5.1).

We  recall an abstract decomposition of $V(\psi)\otimes A$ from an
earlier work of the first author in Section 1.  One of our main
results in Theorem 4.2 where we give an explicit description of
the components of this decomposition in terms of automorphisms
discussed earlier. This  allows us to compute the characters of
the components of $V(\psi)\otimes A$ in terms of the character of
$V(\psi)$.  In Section 5 we discuss in detail  the case $n=2$.

We also note that Theorem~3.4 of \cite{PB} can be easily deduced
from our Proposition~3.4.  The polynomials $P_{h, \lambda}$ are in
fact constant in this case.

\section{Preliminaries}
All vector spaces,  algebras and tensor products
are over complex numbers $\CC$.  Let $\ZZ, \NN$ and $\ZZ_+$ denote
integers, non-negative integers and positive integers,
respectively. For any Lie algebra ${\mathcal A}$ denote
$U({\mathcal A})$ the universal enveloping algebra of ${\mathcal
A}$.

Recall that $\gg$ is an arbitrary Kac-Moody Lie-algebra (see
\cite{K} for details) and  $A = \CC[t_1^{\pm 1}, \cdots, t_n^{\pm
1}]$ is the ring of Laurent polynomials  in $n$ commuting
variables. Write $\hh = \hh' \oplus  \hh^{''}$ for some subspace
$\hh^{''}$ of $\hh$. 

Set

$$\begin{array}{lll}
\gg_{A} & = & \gg^{'} \otimes A \oplus \hh^{''} \supseteq \gg \\
\hh_{A} & = &\hh^{'} \otimes A \oplus \hh^{''} \supseteq \hh.
\end{array}
$$

For $m = (m_1, \cdots, m_n) \in \ZZ^n$ denote $t^m = t_1^{m_1}
\cdots, t^{m_n}_n \in A$.  The Lie bracket in $\gg_A$ is defined as
follows:
$$
\begin{array}{lll}
[X \otimes t^m, Y \otimes t^s ] & = & [X,Y] \otimes t^{m+s}\\[1mm]
[h, X \otimes t^m] & = & [h, X] \otimes t^m \\[1mm]
[h, h'] & = & 0 \ {\rm for}
\end{array}
$$

$X,Y \in\gg', h, h'  \in \hh'', m, s \in  \ZZ^n$.  The Lie algebra
$\gg_A$ is called \emph{Loop Kac-Moody Lie algebra}.

Let $D$ be $\CC$-linear span of derivations $d_1, \cdots, ,d_n$
where $[d_i, \gg' \otimes t^m] = m_i \gg' \otimes t^m$ and $[d_i,
d_j] = [d_i, \hh''] = 0$ for all $i,j=1, \ldots, n$.

Denote $\widetilde{\gg_A} = \gg_A\oplus D,\widetilde{\hh_A} =
\hh_A \oplus D$.  Let $\gg = N^- \oplus \hh \oplus N^+$ be the
 standard triangular decomposition. Then $\gg_A = (N^- \otimes A) \oplus \hh_A
\oplus (N^+ \otimes A)$ is a triangular decomposition of $\gg_A$.

Let $\psi : \hh_A \to \CC$ be any linear map.  Define the induced
map  $\overline{\psi} : \hh_A \to \CC$ as follows:
$$\overline{\psi} (h \otimes t^m) = \psi (h \otimes t^m) t^m$$ for
$h \in \hh', m \in \ZZ^n,$ and $\overline{\psi} (\hh'') = \psi
(\hh'')$. Let $\CC$ be the one dimensional representation of $(N^+
\otimes A) \oplus \hh_A$ where $N^+ \otimes A$ acts trivially and
$\hh_A$ acts via $\psi$: $h\cdot 1=\psi(h)1$ for any $h\in \hh_A$.
Consider the induced module for $\gg_A$,
$$M(\psi) = U(\gg_A){\displaystyle{\bigotimes_{(N^+ \otimes A) \oplus \hh_A}}}\CC.$$
By standard arguments one can show that $M(\psi)$ has a unique
irreducible quotient denoted by $V(\psi)$.

Now we will define a graded version of this module.  We  consider
$A$ as a $\widetilde{\hh}_A$-module with respect to the following
action:
$$\begin{array}{rll}
(h \otimes t^m)\cdot t^k & = & \overline{\psi}(h \otimes t^m) t^k \\
h'\cdot t^m & = & \overline{\psi}(h')t^m \\
d_i\cdot t^m & = & m_i t^m.
\end{array}
$$
for $h \in \hh'$, $h' \in \hh''$, $d_i \in D$, $i=1, \ldots, n$,
$m, k \in \mathbb Z^n$. Extend $\overline{\psi}$ to $U(\hh'
\otimes A)$  by an algebra homomorphism.  Denote the image
$\overline{\psi}(U(\hh' \otimes A))$ by $A_{\psi}$, which is also
a $\widetilde{\hh}_A$-module.

We have the following criteria of irreducibility.

\begin{lemma}\label{(1.2)-Lemma} [\cite{E1}, Lemma 1.2]
The $\widetilde{\hh}_A$-module $A_{\psi}$ is an irreducible  if
and only if each homogeneous element of $A_{\psi}$ is invertible
in $A_{\psi}$.

\end{lemma}

Throughout this paper we assume that $A_{\psi}$ is an irreducible
$\widetilde{\hh}_A$-module. Define
$$Supp \, \overline{\psi}= \{ m \in \ZZ^n \mid (A_{\psi})_m
\not= 0 \}.$$

Clearly,  $Supp \, \overline{\psi}$ is a subgroup of $\ZZ^n$.

Suppose now that $N^+ \otimes A$ acts trivially on $A_{\psi}$.
Consider the induced module for $\widetilde{\gg}_A$,
$$M(\overline{\psi}) = U(\widetilde{\gg}_A)
{\displaystyle{\bigotimes_{(N^+ \otimes A) \oplus
\widetilde{\hh}_A}}}A_{\psi}.$$

Since $A_{\psi}$ is an irreducible module, it follows that
$M(\overline{\psi})$ has a unique irreducible quotient
$V(\overline{\psi})$.  It is standard that $M(\psi)$ and
$M(\overline{\psi})$ are weight module with respect to $\hh$ and
$\hh \oplus D$, respectively, that is
$$M(\psi)=\oplus_{\mu\in \hh^*}M(\psi)_{\mu}$$

and
$$M(\overline{\psi})=\oplus_{\mu\in (\hh \oplus D)^*}M(\overline{\psi})_{\mu}$$
(We refer to \cite{E1}, Section 3  for detail). Since any quotient
of a weight module is a weight module, $V(\psi)$ and
$V(\overline{\psi})$ are irreducible weight modules for $\gg_A$
and $\widetilde{\gg}_A$, respectively.

  Our
first goal is to establish a relationship between
$V(\overline{\psi})$ and $V(\psi)$.  We  define now $\widetilde\gg_A$
 module structure on $V(\psi) \otimes A$ as follows:

$$\begin{array}{rll}
(X \otimes t^m)\cdot (v \otimes t^k) & = & ((X \otimes t^m)\cdot
v)
\otimes t^{m +k} \\
h\cdot (v  \otimes t^k) & = & (h v) \otimes t^k \\
d_i\cdot (v \otimes t^k) & = & k_i (v \otimes t^k)
\end{array}
$$
for $ v \in V(\psi)$, $X \in \gg'$, $h \in \hh''$, $d_i \in D$,
$i=1, \ldots, n$, $m, k \in \ZZ^n$. We have the following

\begin{proposition}\label{(1.2)-Proposition} [\cite{E2},Proposition~3.5] Let $G \subset
\ZZ^n$ be such that $\{ t^m, m \in G\}$ is a set of coset
representatives for $A/A_{\psi}$.  Let $v$ be a highest weight
vector in $V(\psi)$. Set $v(m) = v \otimes t^m$ for any $m\in
\ZZ^{n}$. Then
\begin{enumerate}
\item $V (\psi) \otimes A = {\displaystyle{\bigoplus_{m \in G}}} U
v(m)$ where $Uv(m)$ is the $\widetilde{\gg}_A$-submodule
generated by $v(m)$. \item Each $U v(m)$ is an irreducible
$\widetilde{\gg}_A$-module. \item $U v(0) \cong
V(\overline{\psi})$ as $\widetilde{\gg}_A$-modules. \item All
components in (1) are isomorphic as $\widetilde{\gg}_A$-modules
up to a grade shift, that is the $D$ action is shifted by a vector
in $\CC^n$.
\end{enumerate}

\end{proposition}

We now recall some standard results on $M(\psi)$ and
$M(\overline{\psi})$.  We first observe that $\widetilde{\gg}_A$
is a Pre- exp-polynomial algebra in the sense of [\cite{BGLZ},
Example 1].

\begin{lemma}\label{lemmas 1.4,1.5,1.6,1.7}
\begin{enumerate}
\item \label{(1.4)-Lemma} [\cite{BGLZ},Theorem 2.12] The
$\widetilde{\gg}_A$-module $M(\overline{\psi})$ is  irreducible if
and only if $M(\psi)$ is an irreducible $\gg_A$-module. 
\item \label{(1.5)-Lemma} [\cite{E2},Lemma 3.6] The $\widetilde{\gg}_A$ module
$V(\overline{\psi})$ has finite dimensional weight spaces with
respect to $\hh \oplus D$ if and only if $\gg_A$ module $V(\psi)$ has finite
dimensional weight spaces with respect to $\hh$. 
\item \label{(1.6)-Lemma} [\cite{E2},Lemma 3.7] The $\gg_A$ module $V(\psi)$ has
finite dimensional weight spaces with respect to $\hh$ if and only
if $\psi$ factors through $\hh' \otimes A/I$ for some co-finite
ideal $I$ of $A$. \end{enumerate}
\end{lemma}

\begin{remark} From the Proof of Lemma 3.7 in [E2], we can choose the co-finite
ideal $I$ to be generated by polynomials $P_i$ in the variable
$t_i$ with nonzero constant term, see also \cite{BGLZ}, Theorem
2.9.
\end{remark}

\section{Modules with finite dimensional spaces}
 In this section we discuss modules with
finite dimensional weight spaces.  We have seen in Section 1 that
$V(\psi)$ has finite dimensional weight spaces if there exist
polynomials $P_1, \cdots, P_n$ in variable $t_1, \cdots, t_n$ which
generate a co-finite ideal $I$.   We can assume that each
polynomial is not a constant.  Indeed, otherwise the ideal $I$
coincides with the algebra $A$, $\psi$ is a trivial function and
the corresponding module is one dimensional.

Set $\Gamma=Supp \, \overline{\psi}$. Under our assumption
$\Gamma$ is a subgroup of $\ZZ^n$. (see below Lemma 1.1)

\begin{lemma}\label{(2.1)-Lemma} The rank of $\Gamma$ equals $n$.
\end{lemma}

\begin{proof} Suppose that the rank of $\Gamma$ is less than $n$.  After a change of
variables we can assume that
$$h \otimes t^m\cdot t^k_n v = 0,$$
for all $k \in \ZZ \backslash \{0\}$ and $m\in \ZZ^n$ such that
$m_n = 0$.  On the other hand, there exists a non-constant
polynomial $P_n$ in variable $t_n$ such that $h\otimes t^m\cdot
P_n(t_n) v = 0$. Now it is easy to conclude that $(h \otimes
t^m)\cdot v = 0$ for all $m \in \ZZ^n$. Thus the function $\psi$
is trivial and $V(\psi)$ is a one dimensional module, which is of no interest.
\end{proof}

It follows from Lemma \ref{(2.1)-Lemma} that $$A_{\psi} = \CC
[t_1^{\pm r_1}, \cdots, t_n^{\pm r_n}]$$ and $$\Gamma = r_1 \ZZ
\oplus \cdots, \oplus r_n \ZZ$$ for some positive integers $r_1,
\cdots, r_n$.  Set $R = r_1 r_2 \cdots, r_n$. If $R = 1$ then
$V(\psi) \otimes A \cong V(\overline{\psi})$ as
$\widetilde{\gg}_A$-modules.

From now on we assume that $R \ge 2$.  Then  $V(\psi)\otimes A$
decomposes into $R$ irreducible $\widetilde{\gg}_A$-modules by Proposition 1.2(1). 
We will give a more explicit description of these components in
Section 4.

\begin{proposition}\label{(2.3)-Proposition}   For each $i, 1 \le i
\le n$, there exists a $\gg$-module automorphism $\sigma_i$ of
$V(\psi)$ of order $r_i$.
\end{proposition}

\begin{proof} Fix $i\in \{1, \ldots, n\}$. Let $\xi_i$ be the $r_{i}$-th
primitive root of unity. Define
 an algebra automorphism $\sigma_i$ of $A$ as follows: $\sigma_i (t_i) = \xi_i
t_i$ and $\sigma_i (t_j) = t_{j}$ for $i \not= j$.  It can be
extended to an automorphism of $\gg' \otimes A$ by  the identity
action on $\gg'$. Also, defining the identity action of $\sigma_i$
on $h''$ we obtain an automorphism  of order $r_i$ on $\gg_A$. It
can be further extended to an algebra automorphism of $U(\gg_A)$.
This automorphism, clearly, respects the triangular decomposition
of $\gg_A$. Let $J$ be the left ideal generated by $N^+\otimes A$,
$h \otimes t^m - \psi (h \otimes t^m)$ and $h' -\psi (h')$,
$h\in \hh'$, $h' \in \hh''$ and $ m \in \ZZ^n$.

Suppose that $m_i \not\equiv 0 (r_i)$. Then we have
$$\begin{array}{lll} &&  \sigma_i(h \otimes t^m - \psi(h \otimes t^m)\big) \\
& =  & h \otimes \xi_i^{m_i} t^m \\
&= & \xi^{m_i} (h \otimes t^m - \psi (h \otimes t^m)\big).
\end{array}$$
On the other hand, if $m_i \equiv 0 (r_i)$ then we have
$$\begin{array}{lll}
&& \sigma_i (h \otimes t^m - \psi (h \otimes t^m)) \\
&= & h \otimes \xi_i^{m_i} t^m - \psi (h \otimes t^m) \\
&= & h \otimes t^m - \psi (h \otimes t^m),
\end{array}
$$
as $ \xi_i^{m_i} = 1$. This implies that
 $\sigma_i$ leaves the  ideal $J$ invariant.
 Since $M(\psi)$  can be identified with
$U(\gg_A)/J$, we obtain that $\sigma_i$ induces an automorphism of
$M(\psi)$.

 Suppose now that $N$ is a proper submodule of $M(\psi)$.
Let $\sigma_i ((X_{\beta_1} \otimes t^{m^1}) \cdots,
(X_{\beta_{\ell}} \otimes t^{m^{\ell}}) v) \in \sigma_{i} (N)$, where
$X_{\beta_j} \in \gg', m^1, \cdots, m^{\ell} \in \ZZ^n$.  Then we
have
$$\begin{array}{lll}
& & (X_{\beta}\otimes t^m)\cdot \sigma_i ((X_{\beta_1} \otimes
t^{m^{1}}) \cdots,
(X_{\beta_{\ell}} \otimes t^{ m^{\ell}}) v) \\
& = & \xi_i^{-m_i} \sigma_i ((X_{\beta} \otimes t^{m})
(X_{\beta_1} \otimes t^{m^{1}}) \cdots, (X_{\beta_{\ell}} \otimes
t^{m^{\ell}}) v) \in \sigma_i(N).
\end{array}$$
Thus $\sigma_i(N)$ is also a submodule of $M(\psi)$.  Clearly,
$\sigma_i(N)$ is a proper submodule, as $\sigma_i(v) = v$ for a
highest weight vector $v$. We conclude that  the sum of all proper
submodules $M(\psi)$ is invariant under $\sigma_i$.  Hence,
$\sigma_i$ is an automorphism on $V(\psi)$ of order $r_i$.
\end{proof}

Let $\overline{\Gamma} = \ZZ/_{r_1 \ZZ} \oplus \cdots, \oplus
\ZZ/_{r_n \ZZ} \cong \ZZ^n/\Gamma$. We have the following
immediate corollary.

\begin{corollary}\label{(2.4)-Corollary}
\begin{enumerate}

\item For any $k = (k_1, \cdots, k_n) \in \overline{\Gamma}$  there
exists a $\gg$-module automorphism $\eta_k=\sigma_1^{k_1} \cdots,
\sigma_n^{k_n}$ of $V(\psi)$ of finite order. 
\item For any $k =(k_1, \cdots, k_n) \in \overline{\Gamma}$ there exists a finite
order automorphism $\tau_k$ of $\widetilde{\gg}_A$ such that $\tau_k(X \otimes t^m) = \xi^{k_1 m_1}_1\cdots, \xi_n^{k_n m_n} X
\otimes t^m \  {\rm for} \  X \in \gg', m = (m_1, \cdots, m_n) \in \ZZ^n  \mbox{and} \ 
\tau_k (\gg  \oplus D) = Id.$
\end{enumerate}
\end{corollary}

\section{Irreducible components of $V(\psi)
\otimes A$}  In this section we  give a description of the
components of $V(\psi) \otimes A$ as a $\widetilde{\gg}_A$ module in terms of the automorphisms
constructed in Corollary~\ref{(2.4)-Corollary}.

First, recall the definition of exp-polynomial maps from
\cite{BZ}.

\begin{definition}\label{(3.1)-Definition} A function $f : \ZZ^n \to \CC$ is
called an \emph{exp-polynomial map} if $f$ can be written as a
finite sum
$$f(m_1, \cdots, m_n) = \displaystyle\sum_{\substack{a \in (\CC^*)^n\\ k\in 
\mathbb Z^n}}
C_{k,a} m_1^{k_1}\cdots, m_n^{k_n} a_1^{m_1} \cdots, a_n^{m_n},$$
$C_{k,a}\in \CC$ and $a = (a_1, \cdots, a_n)\in (\CC^*)^n, k=(k_1,\ldots, k_n)$
\end{definition}

Given $h \in \hh'$ define the function $\psi_h: \ZZ^n \to \CC$ as
follows: $$\psi_h(m_1, \cdots, m_n):= \psi(h \otimes t^{m_1}_1
\cdots, t_n^{m_n})$$ for $m_1, \cdots, m_n \in \ZZ$.

\begin{lemma}\label{(3.2)-Lemma} The following conditions are equvilent for 
$\psi \in (\hh_{A})^*.$
\begin{enumerate}
\item[{(1)}] There exist polynomials $P_1, \cdots,P_n$ in variables
$t_1, \cdots,t_n$ such that
\item[{(R)}] $\psi(\hh' \otimes A P_i(t_i))= 0$
\item[{(2)}] $\psi_h$ is an exp-polynomial map for all $h \in \hh'$.
\item[{(3)}] $V(\psi)$ has finite dimensional weight spaces with respect to $\hh$.
\end{enumerate}
\end{lemma}
\begin{proof}
Fix  $h \in \hh'$ and consider $\psi_h$. Taking $\psi_h$ in Lemma 2.7 of $[RZ],(1)\Leftrightarrow(2)$ follows. 
Just note that (2.4) in Lemma 2.7 of $[RZ]$ is precisely our relation $(R)$. 
Also note that the expression 2.5 in Lemma 2.7 of $[RZ]$ is an $\exp$-polynomial map.\\
Now $2\Leftrightarrow3$ follows from Lemma 1.3(3) and Remark 1.
\end{proof}

Given $a=(a_1, \ldots, a_n) \in (\CC^*)^n$, define the
exp-polynomial function $f_a: \ZZ^n \to \CC$ as follows:
$$f_a(m_1, \ldots, m_n)=m_1^{k_1} \cdots m_n^{k_n} a_1^{m_1} \cdots, a_n^{m_n}.$$

\begin{lemma}\label{(3.3)-Lemma} Functions
$f_a$  are linearly independent for different $a \in (\CC^*)^n$.
\end{lemma}

\begin{proof} Follows from [\cite{BZ},Corollary 2.4].
\end{proof}

  Given $\lambda = (\lambda_1, \cdots, \lambda_n)$ define ${\rm
  exp}
\ \lambda: \ZZ^n \to \CC$ by
$${\rm exp} \  \lambda(m_1, \cdots, m_n) = \lambda_1^{m_1}\cdots, \lambda_n^{m_n}.$$
Assume that the image of $\overline{\psi}$ equals $A_{\psi} = \CC
[t_1^{\pm r_1} \cdots, t_n^{\pm r_n} ]$ and $R = r_1 r_2 \cdots,
r_n$.\\
Define $\overline{\Gamma}$ action on $(\CC^*)^n$ as
$k\cdot \lambda = (\xi_1^{k_1}\lambda_1, \cdots,\xi_n^{k_n} \lambda_n)$ where 
$k = (k_1, \cdots, k_n) \in \overline{\Gamma}$ and $\xi_i$ is the $r_i$-th
primitive root of unity. Write  $\xi^k = (\xi_1^{k_1},\cdots, \xi_n^{k_n})$.\\
Let $B$ be a set of coset representatives of this action.\\
Define operators $T_{k}$ such that 
$$T_{k}.\exp(\lambda) = \exp(k.\lambda)\  for \ k \in \overline{\Gamma}.$$
Further let $P_{R}= \displaystyle{\sum_{k \in \overline{\Gamma}}}T_{k}.$\\

\begin{proposition}\label{(3.5)-Proposition}   Suppose $\psi:\hh' \otimes A \to \CC$ is
an exp-polynomial map, that is  the map $\psi_h (m_1, \cdots, m_n)
= \psi (h \otimes t^m)$ is an exp-polynomial  for any $h \in
\hh'$.
 Then we have
$$\psi_h = P_R {\displaystyle{\sum_{\lambda \in B}}} p_{h, \lambda} exp \
\lambda,$$ with $p_{h, \lambda}$ being some polynomial function
of the form:
$$p_{h, \lambda}(m_1, \ldots, m_n) = \sum_{\ell\in \ZZ^n} C_{h, \lambda,\ell} \ m_1^{\ell_1} \cdots,
m_n^{\ell_n},$$ where $\lambda\in B$, $h \in \hh'$, the sum is
finite and $C_{h, \lambda, \ell}$ are constants depending on
$h,\lambda$ and $\ell$.
\end{proposition}

\begin{proof} For each $h \in \hh'$ write the function $\psi_{h}$ in the form:
$$\psi_{h}(m_1, \cdots, m_n) = {\displaystyle{\sum_{\lambda \in (\CC^
*)^n}}} p_{h, \lambda} (m_1, \cdots, m_n) \lambda_1^{m_1} \cdots,
\lambda_n^{m_n},$$ where $p_{h,\lambda}$ is some polynomial map
for every $\lambda\in (\CC^ *)^n$. Then for a fixed $k = (k_1,
\cdots, k_n) \in \overline{\Gamma}$ we have
$$\psi_{h} (m_1, \cdots, m_n) = {\displaystyle{\sum_{\lambda \in (\CC^
*)^n}}} P_{h, k . \lambda} (\xi_1^{k_1} \lambda_1)^{m_1} \cdots, (\xi^{k_n}_n
\lambda_n)^{m_n}.$$

Since $\psi_h(m_1, \cdots, m_n) = 0$ if $m_i \not\equiv 0(r_i)$ for
some  $i$, and since $(\xi_j^{k_j})^{r_j} = 1$ for all $j$, we
have

$$\psi_h(m_1, \cdots, m_n) = {\displaystyle{\sum_{\lambda \in
(\CC^*)}}} p_{h, k .\lambda} \lambda_1^{m_1} \cdots, \lambda_n
^{m_n}.$$

 By Lemma \ref{(3.3)-Lemma}, we have
$$p_{h,\lambda} = p_{h, k . \lambda} \ {\rm for \ all \ } k \in
\overline{\Gamma}.$$ Thus if $\lambda$ occurs in the summation
then $\xi^k \lambda$ also occurs with the same coefficient. Hence,
we have
$$\psi_h = P_R {\displaystyle{\sum_{\lambda \in B}}} p_{h, \lambda} \ {\rm exp}
\ \lambda .$$ Note that in this summation if $\lambda$ occurs then
$k .\lambda$ does not occur (this is the meaning of $B$). Also
note that such an expression for $\psi_h$ need not be unique.
\end{proof}
We would now like to prove that $V(\psi)$ admits a certain tensor product decomposition 
if $\psi$ is a $\exp$-polynomial map. First we will prove certain results on 
$\exp$-polynomial maps. We need to use Lemma 2.7 of $[RZ].$

Let $\psi$ be an $\exp$-polynomial map, $P_1, \cdots,
P_n$ some polynomials in $t_1, \cdots, t_ n$, respectively, such that
$$\psi(h \otimes AP_i(t_i)) = 0,$$
for all $h \in \hh'$, $i=1, \ldots, n$.\\
We can assume that the leading terms of these polynomials equal
$1$. Then we have
$$\psi_h = {\displaystyle{\sum_{\lambda \in (\CC^*)^n}}} p_{h,\lambda}
\ {\rm exp}  \ \lambda.$$

For each $i=1, \ldots, n$, the $i$-th degree of $p_{h,
\lambda}(m_1,\cdots, m_n)$ be the maximal degree of $m_i$. This
degree does not depend on $h$ and only depends on the polynomial
$P_i$.

Suppose $\lambda = (\lambda_1, \cdots, \lambda_n)$ occurs in the
summation above. Then  $\lambda_i$ is a root of $P_i$ for each $i$
by (2.5) in Lemma 2.7 of $[RZ]$. The multiplicity of $\lambda_i$ is the $i$-th
degree of $p_{h, \lambda}$ plus one.    Further, we have seen in
the proof of Proposition 3.4 that if $\lambda_i$ occurs in
the sum then $\xi_i^{k_i}\lambda_i$ also occurs with the same
coefficient. Hence, if $\lambda_i$ is a root of $P_i$ then
$\xi_i^{k_i}\lambda_i$ is also a root of $P_i$ with the same
multiplicity. 
Then we can write
$$P_i (t_i) = {\displaystyle{\prod_{j=1}^{s_i}}} \ \
{\displaystyle{\prod_{\ell=1}^{r_i}}} (t_i - \xi_i^{\ell} a_{ij})^{b_{ij}},$$
for some scalars $a_{ij}, b_{i j}$ and $s_i$ depending on $P_i$.  Further, 
$(a_{ij}/a_{ij'})^{r_i} = 1$ implies $j= j'$.
Define
$$P_{i,\ell}(t_i) = {\displaystyle{\prod_{j=1}^{k_i}}} \ (t_i - \xi_i^{\ell}
a_{ij})^{b_{ij}}.$$
For each $i=1, \ldots, n$ fix $j_i$, $1 \le j_i \le r_i$.
Set $J = (j_1, \cdots, j_n)$ and consider the ideal $I_J$ of $A$
generated by $P_{1, j_1}\cdots, P_{n, j_n}$.

\begin{lemma}\label{(3.5)-Lemma} Let $J' = (j_1', \cdots, j_n')$,  
$1 \le j_i' \le r_i$, $i=1, \ldots, n$.
   Suppose $J \not= J'$.  Then the ideals $I_J$ and
$I_{J'}$ are co-prime, that is $I_J + I_{J'} = A$.
\end{lemma}

\begin{proof} Since $ J \not= J'$ there exists $i, 1 \le i \le n$,
such that $j_i \not= j_i'$.  Then the polynomial $P_{i, j_i}$ and $P_{i, j_i'}$
have no common roots.  Thus the ideal  of $\CC [t_i, t_i^{-1}]$ generated  by $P_{i, j_i}$ and
$P_{i, j_i'}$ coincides with $\CC
[t_i, t_i^{-1}]$, which implies the statement.  
\end{proof}

Chinese Reminder theorem implies immediately the following statement.

\begin{proposition}\label{(3.9)-Proposition} 
Let $$I = {\displaystyle{\prod_{J \in \overline{\Gamma}}}}
I_J = {\displaystyle{\bigcap_{J \in \overline{\Gamma}}}} I_J.$$ Then
$$\gg' \otimes A/I
\simeq  \gg'
\otimes \Big({\displaystyle{\bigoplus_{J \in \overline{\Gamma}}}}A/I_J\Big).$$
\end{proposition}

We also have
   
\begin{lemma}\label{(3.10)-Lemma} [\cite{E2},Remark 3.9]
Let $\psi:\hh_A \to \CC$ be a linear map satisfying
$\psi(\hh' \otimes I) = 0$ for some co-finite ideal $I$ of $A$.  Then
$(\gg' \otimes I) V(\psi)= 0$.
\end{lemma}

Let  $\alpha_1, \cdots, \alpha_k$ be a set of
simple roots of $\gg$. We  will consider a standard  ordering on $\hh^*$: for  $\eta_1, \eta_2\in \hh^*$ we say that $\eta_1
\leq \eta_2$ if and only if $\eta_2 - \eta_1 = \sum_i n_i \alpha_i$ for some  non-negative integers
$n_i's.$

\begin{proposition}\label{(3.11)-Proposition}  Let $I_1$  and $I_2$ be co-prime co-finite
ideals of $A$.  Let $\psi_1$ and $\psi_2$
be linear maps from $\hh_A \to \CC$ such that $\psi_i(\hh'
\otimes I_i) = 0$ for $i = 1, 2$.  Then $$V(\psi_1 + \psi_2) \simeq V(\psi_1)
\otimes V(\psi_2)$$ as $\gg_A$-modules.
\end{proposition}

\begin{proof} We will show first that $V(\psi_1) \otimes V(\psi_2)$ is a cyclic module generated
by $v_1\otimes v_2$ where $v_i$ is a highest weight vector of
$V(\psi_i)$ for $i = 1,2$.

Since $I_1$, and $I_2$ are co-prime, we have
$I_1 + I_2 = A.$
Thus, there exists $f_i \in I_i$, $i=1,2$, such that $f_1 + f_2=1$.
For  $X \in \gg'$ and $h \in A$ consider
$$
\begin{array}{lll}
X f_1 h(v_1 \otimes v_2) & = & v_1 \otimes X f_1 h v_2 \\
& = & v_1 \otimes (X h - X f_2 h) v_2 \\
& = & v_1 \otimes X h v_2,
\end{array}
$$
as $X f_1 h v_1 = X f_2 h v_2 = 0$ by Lemma \ref{(3.10)-Lemma}.
Repeating this process we see that the module generated by $v_1 \otimes v_2$
contains $v_1 \otimes V(\psi_2)$.
Now taking $X f_2 h$ instead of $X f_1 h$ we obtain that the module 
generated by all elements $v_1 \otimes w$, $w \in V(\psi_2)$ contains $V(\psi_1)
\otimes V(\psi_2)$. We conclude that $v_1\otimes v_2$ generates $V(\psi_1) \otimes V(\psi_2)$.  
Hence, $V(\psi_1 + \psi_2)$
is a  homomorphic image of $V(\psi_1) \otimes V(\psi_2)$.
To complete the proof it is sufficient
to show that $V(\psi_1) \otimes V(\psi_2)$ is irreducible.  Let

\begin{equation} \label{(A3)}
v = {\displaystyle{\sum_{\lambda + \mu = \eta, i}}} v_{\lambda i} \otimes
v_{\mu i} \in V(\psi_1) \otimes V(\psi_2)
\end{equation}

  be a
 vector of weight $\eta$.  We can assume that $\{v_{\lambda
i}\}$
 is a linearly  independent set.   Note that $\lambda$ and $\mu$ may occur several
times but with different vectors.  That is why the additional
index $i$ is used.
  Now choose $\mu'$ to be the minimal among the $\mu$'s that occur in \eqref{(A3)}  with
respect to the  ordering defined above.  Fix a weight vector $v_{\mu' s}$ of
weight $\mu'$  in \eqref{(A3)}. 
Then there exists $X
\in U(\gg_A)_{\psi_2 - \mu'}$ such that $X v_{\mu' s} = v_2$ and  $X v_{\mu' i}$ is a multiple of $v_2$ for all $i$.
Using the arguments as above, we obtain that there exists
$X' \in U(\gg_A)$ such that $X' v = {\displaystyle{\sum_{\lambda +
\mu = \eta, \ i}}} v_{\lambda i} \otimes X v_{\mu i}$.  We claim
that $w=X v_{\mu i} = 0$ for $\mu \not= \mu'$.  Indeed, if  $w \neq 0$ then it has the weight   $\psi_2 - \mu' + \mu \le \psi_2$  implying that
$\mu < \mu'$. But this  contradicts the minimality of $\mu'$.  Hence, $X v_{\mu i} = 0$ for  $\mu \not= \mu'$ and any $i$.
  Thus
$$X' v = \sum v_{(\eta - \mu')i} \otimes k_i v_2,$$
for some $k_i\in \CC$. This element
 cannot be zero since $v_{(\eta -\mu')i}$ is a linearly independent set
and at least one term is nonzero.   Thus we
proved that given an element $v$ in $V(\psi_1) \otimes V(\psi_2)$ there
exists $X' \in U(\gg_A)$ such that $X' v = w \otimes v_2$ for some
nonzero  $w\in V(\psi_1)$.  Repeating this argument one easily shows
that the module generated by $w  \otimes v_2$ contains $v_1 \otimes v_2$.
This completes the proof.
\end{proof}

Let $\psi$ be an exp-polynomial map. For any $h\in \hh'$  set 
$$\Phi_h = {\displaystyle{\sum_{\lambda \in B}}} p_{h, \lambda} \ exp \lambda.
$$
Then
 $$\psi_h = P_R {\displaystyle{\sum_{\lambda \in B}}} p_{h,\lambda}
\ exp \ \lambda = {\displaystyle{\sum_{k \in \overline{\Gamma}}}} \xi^k
\Phi_h.$$ 
 Clearly, $\xi^k \Phi_h$ is an exp-polynomial map for any $k \in \overline{\Gamma}$.  We can choose
$\lambda$'s in $\Phi_h$ in such a way that corresponding polynomials are
$P_{1, r_1}, \cdots, P_{n, r_n}$.  So the polynomials corresponding to the exp-
polynomial map $\xi^k\Phi_h$  are  $P_{1, k_1}, \cdots, P_{n, k_n}$ where
$k = (k_1, \cdots, k_n)$. If $k' = (k_1', \cdots, k_n')$ then the ideals $I_k$ and $I_{k'}$ are mutually
co-prime if $k\neq k'$ by Lemma \ref{(3.5)-Lemma}.  Hence, we have the following:

\begin{corollary}\label{(3.13)-Corollary}
$$V(\psi) \simeq {\displaystyle{\bigotimes_{k \in \overline{\Gamma}}}}
V(\xi^k \Phi)$$ as $\gg_A$-modules.
\end{corollary}

We also have the following statement which is of independent
interest.

\begin{lemma}\label{(3.14)-Lemma} $A_{\Phi} = \CC [t_1^{\pm 1}, \cdots,
t_n^{\pm 1}]$.
\end{lemma}

\begin{proof} Clearly we have
$$A_{\psi} \subseteq  A_{\Phi} = \CC [t_1^{\pm s_1}, \cdots, t_n^{\pm s_n}]$$
for some $s_i \mid r_i$.  Write
$$\Phi_h = {\displaystyle{\sum_{\lambda \in B}}} p_{h, \lambda} exp \lambda.$$
By Lemma 3.4 applied to $\Phi_h$, if $\lambda$ occurs in the summation
of $\Phi_h$, then $\xi_i' \lambda$ also occurs where $\xi_i'$ is the $s_i$th
root
of unity.  As $s_i \mid r_i,\ \xi_i'$ is also $r_i$ the root of unity.
That  is a
contradiction to the definition of $B$ if $s_i > 1$.  Thus $s_i =1$ and
the Lemma is proved.
\end{proof}

\section{Explicit description of  components of  $V(\psi)
\otimes A$}\label{section-4}
  In this section we describe explicitly the irreducible components of $V(\psi)
\otimes A$.  Recall that $V(\psi)\otimes A$ is
a $\widetilde{\gg}_A$-module (see  Section 1). 
 Also recall that for each $k \in \overline{\Gamma}$ there exists an automorphism $\eta_k$ of $V(\psi)$ of finite order by
 Corollary \ref{(2.4)-Corollary}(1).

For  $k \in \overline{\Gamma}$ define the map
$$
\widetilde{\eta}_k : V(\psi) \otimes A \to V(\psi) \otimes A$$ induced
 by $\eta_k$ as follows:
$$\widetilde{\eta}_k(v \otimes p(t)) = \eta_k(v) \otimes p(\xi^{-k} t),
$$
for all $v\in V(\psi)$ and $p(t)\in A$.

\begin{lemma}\label{(4.1)-Lemma} 
The map
$\widetilde{\eta}_k$ is an
automorphism of  $V(\psi) \otimes A$ as a $\widetilde{\gg}_A$-module. 
\end{lemma}

\begin{proof}
 We need to check only that $\widetilde{\eta}_k$ is a $\widetilde{\gg}_A$-module homomorphism. 
 For $X \in \gg', m \in \ZZ^n, p(t) \in A$ and  $v \in
V(\psi)$ we have
$$\begin{array}{lll}
&& \widetilde{\eta}_k ((X \otimes t^m) (v \otimes p(t))) \\
& = & \widetilde{\eta}_k(((X \otimes t^m) v) \otimes t^m p(t)) \\
& = & \eta_k((X\otimes t^m) v) \otimes (\xi^{-k} t)^m p(\xi^{-k} t)\\
& = & (X \otimes (\xi^k)^m t^m)\eta_k(v) \otimes \xi^{-km} t^m
p(\xi^{-k} t) \\
& = & (X \otimes t^m)\eta_k(v) \otimes t^m p(\xi^{-k} t).
\end{array}$$
On the other hand, we have
$$\begin{array}{lll}
&& (X \otimes t^m) \widetilde{\eta}_k (v \otimes p(t)) \\
& = & (X \otimes t^m) (\eta_k(v) \otimes p(\xi^{-k} t)) \\
& = & (X \otimes t^m)\eta_k(v) \otimes t^m p(\xi^{-k} t)).
\end{array}
$$
We conclude that $\widetilde{\eta}_k$ is a $\widetilde{\gg}_A$-
module map, which completes the proof.  
\end{proof}

For any $k = (k_1, \cdots, k_n) \in \overline{\Gamma}$ set
$$(V(\psi) \otimes A)_k= \{ v \in V(\psi) \otimes A \mid \widetilde{\eta}_i v =
\xi_{i}^{k_i} v, 1\le i \le n\}.$$
 It follows immediately from  Lemma \ref{(4.1)-Lemma} that $(V(\psi)\otimes A)_k$ is a
$\widetilde{\gg}_A$-module.

\begin{theorem}\label{(4.2)-Theorem}  Let
$v$ be a highest weight vector of $V(\psi)$.  Then
$$U(\widetilde{\gg}_A)(v \otimes t^k) = (V \otimes A)_k, $$
for all $k \in
\overline{\Gamma}$.
\end{theorem}

\begin{proof} Note that $v \otimes 1 \in (V(\psi) \otimes A)_0$ and
$v \otimes t^k \in (V(\psi) \otimes A)_k$ for $k \in \overline{\Gamma}$.
Thus
$$U(\widetilde{\gg}_A)(v \otimes t^k) \subseteq (V(\psi) \otimes A)_k.$$
We also have 
$$\begin{array}{lll}
V(\psi) \otimes A & = & {\displaystyle{\bigoplus_{k \in \overline{\Gamma}}}}
U(\widetilde{\gg}_A)
v \otimes t^k \\
& \subseteq  &  {\displaystyle{\bigoplus_{k \in \overline{\Gamma}}}}
(V(\psi) \otimes A)_k,
\end{array}$$
by  Proposition \ref{(1.2)-Proposition}.
It implies immediately that
$$U(\widetilde{\gg}_A) (v \otimes t^k) = (V \otimes A)_k.$$
\end{proof}

\section{Characters}
One of the main problems in representation theory is to compute the characters of irreducible modules.  We will indicate some partial results in this direction for modules under consideration.

Let $\psi$ be an exp-polynomial map.  Following the previous section we write
$$\psi = {\displaystyle{\sum_{k \in \overline{\Gamma}}}} \xi^k \Phi,$$
where each $\xi^k \Phi$ in an exp-polynomial map. 

 For each $k \in \overline{\Gamma}$ we fix  an automorphism of finite order $\sigma_k$ of $\gg_A$, which exists by  Corollary \ref{(2.4)-Corollary}, (2). 
 Let
$$\gg_A \lr \ End \ (V (\xi^k \Phi))$$
be the representation map.  Then, clearly, 
$$p_{\xi^k} \circ \sigma_{\ell} = p_{\xi^{k+\ell}},$$
for all $k, \ell \in \overline{\Gamma}$.

Since $\sigma_{\ell}$ is $id$ on $\gg$, we obtain the following  isomorphism of $\gg$-modules:
$$V(\xi^k \Phi) \simeq V(\xi^{k+\ell}\Phi).$$
In particular,
$$V(\xi^k \Phi) \stackrel{\in_k}{\simeq} V(\Phi), \ \  \forall k \in
\overline{\Gamma}.$$
Set
$$W={\displaystyle{\bigotimes_{R-times}}} V(\Phi).$$ Then we have
the following isomorphisms of $\gg$-modules: 
$$V(\psi) \stackrel{\Omega}{\simeq} {\displaystyle{\bigotimes_{k \in
\overline{\Gamma}}}} \ V(\xi^k \Phi)\stackrel{\otimes \in_k}{\simeq} W.$$
Let $\gamma = \otimes \in_k \circ \Omega$.  It is easy to check that
$\gamma \circ  \eta_k \circ \gamma^{-1}: W \to W$ is a $\gg$-module map ($\eta_k$ is defined  in Corollary \ref{(2.4)-Corollary}, (1)).

Since
$$\xi^{\ell} \psi = {\displaystyle{\sum_{k \in\oG}}} \ \xi^{k+\ell} \Phi,$$
for any $\ell \in
\overline{\Gamma}$, we see that $\eta_{\ell}$ induces the isomorphism
$${\displaystyle{\bigotimes_{k \in \oG}}} V (\xi^k \Phi) \simeq
{\displaystyle{\bigotimes_{k \in \oG}}} V(\xi^{k+\ell} \Phi).$$ 
Hence,  $\eta_{\ell}$ and  $\gamma \circ \eta_{\ell} \circ
\gamma^{-1}$ 
permute the factors of the tensor product.  

\begin{proposition} Let ${\rm ch}\ V(\psi)$ and 
 ${\rm ch} \  V(\Phi)$ be the characters of   $V(\psi)$ and $V(\Phi)$, respectively. Then
$${\rm ch} \ V(\psi) = ({\rm ch}  \ V(\Phi))^R.$$
\end{proposition}

\begin{proof} Note that we have the following isomorphism of $\gg$-modules:
$$V(\psi) \simeq  {\displaystyle{\bigotimes
_{R-times }}} V(\Phi).$$  In particular, this is an isomorphism of $\hh$-modules. Hence, 
their characters are same and the statement follows.
\end{proof}
\noindent
{\bf Character of $(V(\psi)\otimes A)_k$ :} Now we will indicate how to compute the character of
$(V(\psi)\otimes A)_k$.  Since all the components are isomorphic upto a grade shift,
it is sufficient to compute the character of $(V(\psi)\otimes A)_0$.
Recall that
$$(V(\psi) \otimes A)_0= \{v \in V(\psi) \otimes A_0 \mid \widetilde{\eta}_k
v = v, \  \forall k \in \oG\}.$$
Let $k = (k_1, \cdots, k_n)\in \oG$ and $\eta_k
= \sigma_1^{k_1} \cdots \sigma_n^{k_n}$. Hence, it is sufficient
to describe fixed points of all $\widetilde{\sigma}_{\ell}$. Set
$$V(\psi)_{(k_1,\cdots, k_n)} = \{v \in V(\psi) \mid \sigma_i v =
\xi_{i}^{k_i}v, 1 \le i \le n\}.$$
Then we have
$$(V(\psi) \otimes A)_0 = {\displaystyle{\bigoplus_{k = (k_1, \cdots, k_n)\in
\ZZ^n}}} \ V(\psi)_{(k_1, \cdots, k_n)}\otimes t^k.$$
Therefore, to compute the character of $(V(\psi) \otimes A)_0$ it is sufficient  to
compute the characters of $V(\psi)_{(k_1, \cdots, k_n)}, k\in \oG$.

  Let $V$ be an
$\NN^d$-graded vector space for some positive integer $d$.  Here $\NN$
denotes the non-negative integers and $\NN^d$ denotes $d$ copies of $\NN$.

Let $V = {\displaystyle{\bigoplus_{\alpha \in \NN^d}}} V (\alpha)$, a direct sum of graded components.  We assume
that each $dim V({\alpha})$ is finite.  Fix a positive integer $r$. Consider
$V^r = V \otimes \cdots, \otimes V$ ($r$ times).  
Then $V^r  = {\displaystyle{\bigoplus_{\alpha \in \NN^d}}}
V^r(\alpha)$, and
$$V^r(\alpha) = {\displaystyle{\bigoplus_{\sum \alpha_i = \alpha}}}
\ V(\alpha_1) \otimes \cdots \otimes V(\alpha_r).$$
Clearly,  ${\rm dim}V^r (\alpha) < \infty$.

Define $\sigma_r : V^r \to V^r$ as follows:
$$\sigma_r (v_1 \otimes \cdots v_r) = v_r \otimes v_1 \otimes
\cdots \otimes v_{r-1},$$
for all $v_1, \ldots, v_r\in V$.

We have that $\sigma_r (V^r (\alpha)) = V^r (\alpha)$, 
for any $\alpha \in \NN^d$. For each $k\in \mathbb Z$
denote
$$V^r_k = \{v \in  V^r\mid \sigma_{r} (v) = \xi^k v \},$$
where $\xi$ is a primitive $r$-th root of unity.  Then 
$$V^r = {\displaystyle{\bigoplus_{k \in \ZZ/r\ZZ}}} V^r_k,$$
where
$$V^r_k = {\displaystyle{\bigoplus_{\alpha \in \NN^d}}} V_k^r(\alpha).$$
For any $\NN^d$-graded vector space $U = {\displaystyle{\bigoplus_{\alpha
\in \NN^d}}} U(\alpha)$ define
$$P_U(X) \ =  {\displaystyle{\sum_{\alpha \in \NN^d}}} \dim \ U(\alpha)  X^{\alpha}
\in \NN[[X_1, \cdots X_d]]$$
where $X^{\alpha} = X^{\alpha_1}_1  \cdots X_d^{\alpha_d}$.

Set
$$C_r(n) = {\displaystyle{\sum_{t \mid gcd (r,n)}}} t \mu (r/t),$$
where $\mu$ is the Mobius function.  

We now recall the following formula
from  \cite{W}.  It is proved in \cite{W} only for $d=1$ but the same proof holds  for
 any  $d$.

\begin{theorem}\label{(5.3)-Theorem} [ \cite{W},Theorem 5.5]
$$P_{V^r_k} (X) = \frac{1}{r} {\displaystyle{\sum_{t\mid r}}} C_t (r) (P_V
(X^t))^{\frac{r}{t}}.$$
\end{theorem}

We will now indicate how to compute the character for $(V(\psi)\otimes
A)_0$ for the case $n = 2$ (hence, $R=r_1 r_2 $).  General case can be treated similarily but it requires 
more complicated  notation.

Recall that $V(\psi)$ is a highest weight module for $\gg_A$ and $V(\psi) \simeq {\displaystyle{\bigotimes_{R-times}}}
V(\Phi)$.  Then $V(\Phi)$ is a highest weight
module for $\gg_A$.  Let $d$ be the rank of the Kac-Moody Lie
algebra  $\gg$.  Let $\alpha_1, \cdots, \alpha_d$ be the simple
roots of $\gg$,  $Q^+$ be the set consisting of all non-negative linear combinations of roots
$\alpha_1, \ldots, \alpha_d$.  Then $Q^+$ can be identified with $\NN^d$. We have 
that
$$V(\Phi)= {\displaystyle{\bigoplus_{\alpha \in \NN^d}}} \ V(\Phi)_{\lambda-
\alpha},$$
where $ \Phi\mid_\hh = \lambda.$

Denote $V = V(\Phi)$ and $V(\alpha) = V(\Phi)_{\lambda-\alpha}$.

Let
$$\begin{array}{lll}
W & = & V \otimes \cdots, \otimes V \ {(r_1 \ {\rm times}}) \ {\rm and}\\
V(\psi) & = & W \otimes \cdots, \otimes W \ (r_2 \ {\rm times}).
\end{array}
$$
Note that each space $V(\xi^k \Phi)$ is identified with $V$ and $\sigma_{1},\sigma_{2}$
are permutations on $V(\psi)$.
Now it is easy to see that the sum $\sum \xi^k \Phi$ can
be rearranged in such way that $\sigma_{1}$ leaves each $W$ invariant and 
$$\sigma_1 (v_1 \otimes \cdots, \otimes v_{r_1}) = v_{r_1} \otimes v_1 \otimes
\cdots, \otimes v_{r_1-1},$$  
for all
$v_i \in V$.  This  extends  to $V(\psi)$  leaving each component
$W$ invariant.  Further we can assume
$$
\begin{array}{lll}
&& \sigma_2(w_1 \otimes \cdots, \otimes \cdots, w_{r_2}) \\
& = & w_{r_2} \otimes w_1 \otimes \cdots, \otimes w_{r_2-1}, w_i \in W.
\end{array}
$$

Since
$$V^{r_1}_{k_1} = W_{k_1} = \{v \in V^{r_1} \mid \sigma_1 (v) = \xi_1^{k_1}
v\},$$
where $\xi_1$ is a $r_i$-th	 primitive root of unity, then we have
\begin{equation}\label{eq-5.4}
P_{W_{k_1}}(X) = \frac{1}{r_1} {\displaystyle{\sum_{t\mid
r_1}}} C_t(r_1) (P_V(X^t))^{r_1/t},
\end{equation}
by  Theorem \ref{(5.3)-Theorem}.

Now we have that $W = {\displaystyle{\bigoplus_{(\alpha, s) \in \NN^{d+1}}}}
W(\alpha)_s$, where $W(\alpha)_s$ is the $s$'s eigenspace of
$\sigma_1$.  Note that $W$ is graded by $\NN^{d+1}$.  Again by
  Theorem \ref{(5.3)-Theorem} we have:

\begin{equation}\label{eq-5.5}
P_{W_{k_2}^{r_2}}(X) = P_{V(\psi)_{k_2}}(X) = \frac{1}{r_2}
{\displaystyle{\sum_{t\mid r_2}}} C_t(r_2) P_W(X^t)^{\frac{r_2}{t}}.
\end{equation}

Formula \eqref{eq-5.5} implies that
$\dim  V(\psi)_{(k_1, k_2)} (\alpha)$ can be computed in terms of
$\dim \ W_{k_1}(\alpha)$. On the other hand, 
$\dim  W_{k_1}(\alpha)$ can be
computed in terms of $\dim V(\alpha)$ by \eqref{eq-5.4}.  Hence, if we know
 the character
of $V(\Phi)$ then we can compute the character of any $V(\psi)_{(k_1,{k_2})}$.
This in turn allow us to compute the character of $(V(\psi)\otimes A)_0$.

We refer to \cite{G1} and \cite{G2} for character formulas for 
graded integrable modules.

\section{Acknowledgment}
\noindent The second author is supported in part by the CNPq grant
(processo 301743/2007-0) and by the Fapesp grant (processo
2005/60337-2).

\end{document}